\documentclass[11pt]{amsart}

\usepackage{amsmath,amssymb,amsthm,amsrefs,a4wide}
\usepackage{latexsym}
\usepackage{indentfirst}
\usepackage{placeins}
\usepackage{booktabs}
\usepackage{algorithm}
\usepackage{algorithmic}
\usepackage{multirow}
\usepackage{color,xcolor,colortbl}

\usepackage{subfigure}
\usepackage{graphicx,color}
\usepackage{diagbox}

\usepackage{hyperref}

\usepackage{cleveref}
\usepackage{braket}

\theoremstyle{plain}
\newtheorem{theorem}{Theorem}
\newtheorem{lemma}[theorem]{Lemma}

\theoremstyle{definition}

\theoremstyle{definition}

\theoremstyle{remark}

\usepackage{qcircuit}

\renewcommand{\d}{\mathrm{d}}

\newcommand{\eps}{\epsilon}

\newcommand{\bbm}{\begin{bmatrix}}
\newcommand{\ebm}{\end{bmatrix}}
\newcommand{\RR}{\mathbb{R}}

\newcommand{\ZZ}{\mathbb{Z}}

\newcommand{\TT}{\mathbb{T}}

\renewcommand{\i}{\mathrm{i}}

\newcommand{\rd}{\,\mathrm{d}}

\newcommand{\half}{\frac{1}{2}}

\newcommand{\mo}[1]{{O}\left(#1\right)}

\newcommand{\la}{\lambda}
\DeclareMathOperator{\dist}{dist}
\newcommand{\Is}{X_*}
\newcommand{\Ie}{X_e}

\newcommand{\res}{\text{res}}

\begin{document}

\title[A note on spike localization for line spectrum estimation]{A note on spike localization for line spectrum estimation}

\author[]{Haoya Li} \address[Haoya Li]{Stanford University,
  Stanford, CA 94305} \email{lihaoya@stanford.edu}

\author[]{Hongkang Ni} \address[Hongkang Ni]{Stanford University,
  Stanford, CA 94305} \email{hongkang@stanford.edu}

\author[]{Lexing Ying} \address[Lexing Ying]{Stanford University,
  Stanford, CA 94305} \email{lexing@stanford.edu}

\thanks{The authors are listed in alphabetical order. The authors thank Wenjing Liao for the helpful discussions.}


\begin{abstract}
  This note considers the problem of approximating the locations of dominant spikes for a probability measure from noisy spectrum measurements under the condition of residue signal, significant noise level, and no minimum spectrum separation. We show that the simple procedure of thresholding the smoothed inverse Fourier transform allows for approximating the spike locations rather accurately.
\end{abstract}

\maketitle

\section{Introduction}\label{sec:intro}

In this note, we consider the problem of recovering the support of dominant spikes in an unknown probability measure $f_*$ on $\RR$. Let $f_*$ be a probability measure of the form
\begin{equation}\label{eq:fs}
  f_*(x) = \sum_{s=1}^S w_s \delta_{x_s}(x) + r(x)
\end{equation}
where the following conditions hold:
\begin{itemize}
\item $S$ is the total number of dominant spikes, $x_s$ are the spike locations, and all weights $w_s$ are bounded by $\beta>0$ from below, i.e., $w_s\ge \beta>0,~1\le s\le S$;
\item $r$ is a residue where its total variation is upper bounded by a constant $\omega$ less than $\beta$, i.e., $r(\RR)=\int_\RR \d r(x)\le\omega<\beta$.
\end{itemize}
The observed data is a signal $y(k)$ for $k\in [-K,K]$ that satisfies $|y(k) - \hat{f}_*(k)|\le \eps$, where $\hat{f}_*(k)=\int_\RR e^{-2\pi \i x k} f_*(x) \d x$ is the Fourier transform of $f_*$. The goal is to recover the support $\Is\equiv \{x_s\}$ of the spikes with a certain accuracy. We are particularly interested in the following question: for what $\eps$ can we estimate the spike locations $x_s$ within precision $\mo{1/K}$? In this note, we show that when $\eps =\mo{\beta-\omega}$, a simple routine estimates the spike locations within distance $\mo{K^{-1}\log(1/(\beta-\omega))}$.

In practice, it is more relevant to consider a periodic version of the problem, in which case $f_*$ is a probability measure on the periodic interval $\TT\equiv[-\frac{1}{2},\frac{1}{2}]$ that satisfies the conditions listed above, except that the residue $r$ satisfies $r(\TT)\le\omega<\beta$ instead. The observed data is a vector $(y(k))_{k=-K}^K$ such that $|y(k) - \hat{f}_*(k)|\le \eps$ for $k\in \{-K,\ldots,K\}$, where $\hat{f}_*=\int_\TT e^{-2\pi \i x k} f_*(x) \d x$ is the Fourier series coefficients of $f_*$. We show that a similar result can be obtained: when $\eps =\mo{\beta-\omega}$ we can estimate the spike locations within distance $\mo{\log(1/(\beta-\omega))/K}$ as long as $K\ge\Theta(\log(1/(\beta-\omega)))$.

This periodic setting is of more practical importance in signal processing. For example, estimating the low-lying eigenvalues $\{\la_s\}_{s=1}^S$ of a Hamiltonian $H$ is an important problem in quantum computing \cites{o2019quantum,somma2019quantum}, where one applies $e^{\i H t}$ to a $\ket{\psi}=\sum_{s=1}^{S}c_s\ket{\psi_s}+c_\res\ket{\psi_\res}$ that is an imperfect superposition of eigenstates $\ket{\psi_s}$, and the residue $c_\res\ket{\psi_\res}$ is orthogonal to all $\ket{\psi_s}$'s. The data obtained is the noisy Fourier transform at frequency $k$ of  $f_*(x) \equiv \sum_{s=1}^S|c_s|^2\delta_{\la_s}(x) + r(x)$, where $r(x)$ corresponds to the residue $c_\res\ket{\psi_\res}$ in $\ket{\psi}$. The recovered support then becomes an estimation of the eigenvalues $\{\la_s\}_{s=1}^S$. 

\subsection{Related work}
The problem considered here is a special version of the line spectrum estimation problem. Many algorithms have been proposed for this important task in signal processing, dating back to Prony's method \cite{prony1795essai}. Prony's method recovers the spectrum exactly using noiseless signals, but it is known to be unstable in the presence of noise. More robust algorithms have been developed, including the matrix pencil method \cite{hua1990matrix}, MUSIC \cite{schmidt1986multiple}, ESPRIT\cite{roy1989esprit} and other subspace methods \cite{cadzow1988signal}. The modern convex relaxation approach, optimizing the $\ell_1$, total variation, and atomic norms, has been extensively developed in \cites{robust_SR,super-res,bhaskar2013atomic,tang2013compressed,tang2014near,li2015off}, to name a few.

The most active area of theoretical analysis in line spectrum estimation is super-resolution \cites{robust_SR, super-res, morgenshtern2016super, izacard2019learning, fernandez2018demixing, fernandez2016super, fernandez2013support, fernandez2013super, tang2013compressed, schiebinger2018superresolution, fannjiang2011exact, fannjiang2012super, liao2015music, liao2016music, li2019conditioning, li2020super, li2021stable, li2022stability, demanet2013super, batenkov2021super, akinshin2015accuracy, katz2022decimated, batenkov2019rethinking, batenkov2018stability, chen2013spectral, da2020compressed, li2019stable, morgenshtern2022super, donoho1992super}, where the goal is to recover the spectrum when the minimal separation is smaller than the Rayleigh distance $\mo{1/K}$. The first work on super-resolution stability was \cite{donoho1992super}, where Donoho introduced the concept of the Rayleigh index and demonstrated the connection between the Rayleigh index, super-resolution factor (SRF), and the allowed perturbation size. In \cites{robust_SR,super-res}, Candes and Fernandez-Granda demonstrated that the convex relaxation method could recover the spike locations in both noiseless and noisy situations if the spikes are separated by $2/K$. The optimal separation result  ($\approx 1/K$) was obtained by Moitra in \cite{moitra2015super}. The convex relaxation approach is shown in \cite{morgenshtern2016super} to achieve near-optimal worst-case performance. A lower bound for reconstruction error is given in \cite{akinshin2015accuracy}, and the minimax error rates for reconstruction have been obtained in \cite{batenkov2021super}. The application of MUSIC and ESPRIT in this setting have been investigated in \cites{liao2016music,li2020super}. The recent literature on super-resolution is vast, and we refer interested readers to review papers and tutorials such as \cite{chi2016tutorial}.

Among these works, several papers have discussed the recovery of positive spikes. For example, \cite{morgenshtern2016super}, and \cite{morgenshtern2022super} analyzed individual spike recovery errors in terms of the super-resolution factor under Rayleigh regularity assumptions. The result in \cite{schiebinger2018superresolution} shows that the spectrum can be exactly recovered without assuming spectral gaps if the observed signal is a noiseless superposition of certain point spread functions, and in particular, Gaussian point spread functions. The BLASSO algorithm is shown in \cite{denoyelle2015support} to be able to recover positive sources if the noise is of the order $O(\Delta^{2S-1})$, where $S$ is the number of spikes and $\Delta$ is the spectral gap.

\subsection{Comparison}
The setting of this note is somewhat different from the setup considered in the majority of works mentioned above: we do not assume any separation in the support $\Is$ or any known structure of the noise. In addition, the total mass of the residue $r(x)$ can approach $\beta$. These conditions make our problem more difficult, thus preventing the direct application of many existing algorithms based on sufficient support separation. As a result, the recovery criteria pursued in this note is somewhat weaker than the ones in the related works above: we ask the estimated support $\Ie$ to be close to $\Is$, i.e., $\max_{x\in\Ie}\dist(x, \Is)$ and $\max_{x\in\Is}\dist(x, \Ie)$ are required to be small.

We would like to point out that the conditions for \eqref{eq:fs} are crucial to the tractability of the problem considered in this note. First, it is important that the measure $f_*$ is positive. Otherwise, it would be impossible to identify nearby spikes with opposite amplitudes since they would cancel under noise and become invisible in the signal. It is also key to ensure that $w_s$ has a positive lower bound $\beta$ since small spikes are impossible to detect from the residue and the noise. Finally, rather than asking for the individual spike locations $\{x_s\}$, one can only seek an estimation of the whole support $\Is$ since it is impossible to distinguish between arbitrarily close (even positive) spikes under noise.

The rest of the paper is organized as follows. \Cref{sec:realline} considers the problem of recovering a measure on $\RR$. In \Cref{sec:periodic}, we deal with the recovery in the interval case.

\section{Real line case}\label{sec:realline}

This section considers the real line case. We consider a simple method that utilizes Gaussian smoothing and inverse Fourier transform. Denote by $\phi$ and $\hat{\phi}$ the Gaussian density function and its Fourier transform:
\begin{equation}\label{eq:gauss}
  \phi(x) = \frac{1}{\sigma/K} \exp\left(-\pi \frac{x}{(\sigma/K))^2} \right), \quad
  \hat{\phi}(k) = \exp(-\pi (k \sigma/K)^2)
\end{equation}
where $\sigma = \sqrt{\frac{1}{\pi}\log\frac{6}{\beta-\omega}}$. Let us introduce
\begin{equation}\label{eq:Iegaussreal}
  \Ie = \left\{x:\left|\int_{|k|\le K}y(k)\hat{\phi}(k)e^{2\pi\i kx}\rd k\right|>\frac{(2\beta+\omega)K}{3\sigma}\right\}.
\end{equation}
The following theorem states that $\Is$ and $\Ie$ are close.
\begin{theorem}
  Suppose $\eps\le \frac{\beta-\omega}{6}$. Then $\Is\subset \Ie$ and for any $x\in \Ie$, $\dist(x, \Is)\le \tau/K$ with $\tau = \frac{1}{\pi}\log\frac{6}{\beta-\omega}$.
\end{theorem}

\begin{proof}
Since $|y(k)-\hat{f_*}(k)|\le \eps$ for $|k|\le K$, direct calculations shows that for any $x$
\begin{equation}\label{eq:approxconvreal}
\begin{aligned}
&\left|\int_{|k|\le K}y(k)\hat{\phi}(k)e^{2\pi\i kx}\rd k-\phi*f_*(x)\right|\\
&=\left|\int_{|k|\le K}(y(k)-\hat{f_*}(k))\hat{\phi}(k)e^{2\pi\i kx}\rd k + \int_{|k|> K}\hat{f}_*(k)\hat{\phi}(k)e^{2\pi\i kx}\rd k\right|     \\
&\le \left|\int_{|k|\le K}(y(k)-\hat{f_*}(k))\hat{\phi}(k)e^{2\pi\i kx}\rd k\right| + \left|\int_{|k|> K}\hat{f}_*(k)\hat{\phi}(k)e^{2\pi\i kx}\rd k\right|\\
&\le \eps \int_{|k|\le K}|\hat{\phi}(k)|\rd k + \int_{|k|>K}|\hat{\phi}(k)|\rd k < \eps \int_{-\infty}^{\infty}\exp(-\pi (k \sigma/K)^2)\rd k + \frac{K}{\sigma}\exp(-\pi\sigma^2)\\
&\le \eps\frac{K}{\sigma} + \frac{K}{\sigma}\exp(-\pi\sigma^2).
\end{aligned}
\end{equation}
\textbf{Step 1.} We first show that $\Is\subset \Ie$ with the help of \eqref{eq:approxconvreal}.
For any $x_s\in\Is$, 
\[
\begin{aligned}
\left|\int_{|k|\le K}y(k)\hat{\phi}(k)e^{2\pi\i kx_s}\rd k\right|&\ge (\phi*f_*)(x_s) - \left|\int_{|k|\le K}y(k)\hat{\phi}(k)e^{2\pi\i kx_s}\rd k-\phi*f_*(x_s)\right|\\
&\ge \frac{K}{\sigma}\beta - \eps\frac{K}{\sigma} - \frac{K}{\sigma}\exp(-\pi\sigma^2) \ge \frac{K}{\sigma}\left(\beta-\frac{\beta-\omega}{6}-\frac{\beta-\omega}{6}\right)\\
&=\frac{K}{\sigma}\left(\frac{2\beta+\omega}{3}\right).
\end{aligned}
\]
Thus $x_s\in \Ie$ by definition \eqref{eq:Iegaussreal}. 

\textbf{Step 2.} Now let us show that for any $x\in \Ie$, $\dist(x, \Is)\le \tau/K$. Using proof by contradiction, let us assume that there exists $x\in \Ie$ that violates this. Then
\begin{equation}\label{eq:step2gaussreal}
\begin{aligned}
\frac{(2\beta+\omega)K}{3\sigma}&\le \left|\int_{|k|\le K}y(k)\hat{\phi}(k)e^{2\pi\i kx}\rd k\right|\\
&\le (\phi*f_*)(x)+ \left|\sum_{|k|\le K}y(k)\hat{\phi}(k)e^{2\pi\i kx}-\phi*f_*(x)\right|\\
&<\frac{K}{\sigma}\exp(-\pi\frac{\tau^2}{\sigma^2}) + \frac{K}{\sigma}\omega + \eps\frac{K}{\sigma} + \frac{K}{\sigma}\exp(-\pi\sigma^2)\\
&\le \frac{K}{\sigma}\left(\frac{\beta-\omega}{6} + \omega + \frac{\beta-\omega}{6} + \frac{\beta-\omega}{6}\right)
= \frac{(\beta+\omega)K}{2\sigma},
\end{aligned}
\end{equation}
which leads to $\frac{\beta+\omega}{2} > \frac{2\beta+\omega}{3} $ and contradicts with $\beta>\omega$. Here we have used \eqref{eq:approxconvreal} in the third line of \eqref{eq:step2gaussreal}.
\end{proof}

\section{Periodic interval case}\label{sec:periodic}

This section considers the case of the periodic interval $\TT$. Introduce the periodic Gaussian function
\begin{equation}\label{eq:pgauss}
  \phi_p(x) = \sum_{j\in\ZZ}\phi(x+j)=\sum_{j\in\ZZ} \frac{1}{\sigma/K} \exp\left(-\pi \frac{(x+j)^2}{(\sigma/K)^2} \right),\quad x\in\TT.
\end{equation}
Notice that $\int_0^1\phi_p(x)\rd x = 1$. Its Fourier coefficients are given by
\begin{equation}\label{eq:pgausshat}
  \hat{\phi}_p(k) = \exp(-\pi (k \sigma/K)^2),\quad k\in\ZZ. 
\end{equation}
The following lemma bounds $\phi_p$ in terms of $\phi$ (defined in \eqref{eq:gauss}).

\renewcommand{\kappa}{10^{-4}}
\begin{lemma}\label{lem:gaussian}
    If $K\ge 3\sigma$, then $\phi_p(x)$ is increasing on $[-\half,0]$ and decreasing on $[0,\half]$. For $x\in [-\frac{1}{3},\frac{1}{3}]$,
    \begin{equation}
        \phi(x) \le \phi_p(x) \le (1+\kappa) \phi(x),\label{ineq:gauss bound}.
    \end{equation}    
    In particular, 
    \begin{equation}
        \sum_{k\in\ZZ}\hat{\phi}_p(k) = \phi_p(0) \le (1+\kappa) \phi(0) = \phi_p(0) \le (1+\kappa)\frac{K}{\sigma}.\label{ineq:p0bound}
    \end{equation}
\end{lemma}
\begin{proof}
  Without loss of generality, we assume $x\ge 0$ since $\phi_p$ is an even function. Notice that the Jacobi theta function $\Theta(x \mid \i\sigma^2/K^2)=\sum_{k\in\ZZ}\exp(-\pi(k\sigma/K)^2+2\pi\i kx)$ is just the Fourier series of $\phi_p(x)$, and we can use its product form
  \begin{equation*}
    \phi_p(x)=\Theta(x \mid \i\sigma^2/K^2)=\prod_{n=1}^{\infty}\left(1-q^{2 n}\right)\left(1+2 q^{2 n-1} \cos (2\pi x)+q^{4 n-2}\right),
  \end{equation*}
  where $q = e^{-\pi\sigma^2/K^2}$. Therefore, $\phi_p(x)$ is decreasing on $[0,\half]$ since each product term is positive and decreasing in $x$.
  
  When $x\in[0,\frac{1}{3}]$, $\phi(x) \le \phi_p(x)$ is trivial from the definition of $\phi_p$. The other direction of \eqref{ineq:gauss bound} can be established using the following calculation.
  \begin{equation*}
    \begin{aligned}
      \frac{\phi_p(x)}{\phi(x)} &= \sum_{j\in\ZZ} \frac{\phi(x+j)}{\phi(x)} = \sum_{j\in\ZZ} \exp(-\pi \frac{K^2}{\sigma^2}(j^2+2xj))\le \sum_{j\in\ZZ} \exp(-9\pi(j^2+2xj))\\
      & = 1+\exp(-9\pi(1-2x))+\sum_{j=1}^{\infty} \exp(-9\pi(j^2+2xj))+\sum_{j=2}^{\infty} \exp(-9\pi(j^2-2xj))\\
      &\le 1+\exp(-9\pi(1-2/3))+\sum_{j=1}^{\infty} \exp(-9\pi j)+\sum_{j=2}^{\infty} \exp(-9\pi j)\\
      &\le 1+10^{-4},
    \end{aligned}
  \end{equation*}
  which completes the proof.
\end{proof}

Let us introduce
\begin{equation}\label{eq:Iegauss}
    \Ie = \left\{x:\left|\sum_{|k|\le K}y(k)\hat{\phi}_p(k)e^{2\pi\i kx}\right|>\frac{6\beta+5\omega}{11}\phi_p(0)\right\},
\end{equation}
where $\sigma = \sqrt{\frac{1}{\pi}\log\frac{12}{\beta-\omega}}$. The following theorem states that $\Is$ and $\Ie$ are close in the periodic case.

\begin{theorem}
  Suppose $\omega<\beta$ and $\eps\le \frac{\beta-\omega}{3}$. Then $\Is\subset \Ie$ and for any $x\in \Ie$, $\dist(x, \Is)\le \tau/K$ with $\tau = \frac{1}{\pi}\log\frac{12}{\beta-\omega}$ for any $K\ge 3\tau$.
\end{theorem}

\begin{proof}
Since $|y(k)-\hat{f_*}(k)|\le \eps$ for $|k|\le K$, direct calculations show that for any $x$
\begin{equation}\label{eq:approxconv}
\begin{aligned}
&\left|\sum_{|k|\le K}y(k)\hat{\phi}_p(k)e^{2\pi\i kx}-(\phi_p*f_*)(x)\right|\\
&= \left|\sum_{|k|\le K}(y(k)-\hat{f_*}(k))\hat{\phi}_p(k)e^{2\pi\i kx} + \sum_{|k|> K}\hat{f}_*(k)\hat{\phi}_p(k)e^{2\pi\i kx}\right|     \\
&\le \left|\sum_{|k|\le K}(y(k)-\hat{f_*}(k))\hat{\phi}_p(k)e^{2\pi\i kx}\right| + \left|\sum_{|k|> K}\hat{f}_*(k)\hat{\phi}_p(k)e^{2\pi\i kx}\right|\\
&\le \eps \sum_{|k|\le K}|\hat{\phi}_p(k)| + \sum_{|k|>K}|\hat{\phi}_p(k)| < \eps \sum_{k\in\ZZ}\hat{\phi}_p(k) + \int_{|k|>K}\exp(-\pi (k \sigma/K)^2)\rd k\\
&\le \eps\phi_p(0) + \frac{K}{\sigma}\exp(-\pi\sigma^2)<(\eps + \exp(-\pi\sigma^2))\phi_p(0).
\end{aligned}
\end{equation}
where the last step uses \eqref{ineq:p0bound}.

\textbf{Step 1.} We first show that $\Is\subset \Ie$ with the help of \eqref{eq:approxconv}.
For any $x_s\in\Is$, 
\[
\begin{aligned}
\left|\sum_{|k|\le K}y(k)\hat{\phi}_p(k)e^{2\pi\i kx_s}\right|&\ge (\phi_p*f_*)(x_s)-\left|\sum_{|k|\le K}y(k)\hat{\phi}_p(k)e^{2\pi\i kx_s}-(\phi_p*f_*)(x_s)\right|\\
&\ge \beta\phi_p(0) - (\eps + \exp(-\pi\sigma^2))\phi_p(0)\\
&\ge \left(\beta-\frac{\beta-\omega}{3}-\frac{\beta-\omega}{12}\right)\phi_p(0)
>\left(\frac{6\beta+5\omega}{11}\right)\phi_p(0),
\end{aligned}
\]
thus $x_s\in \Ie$ by definition \eqref{eq:Iegauss}. 

\textbf{Step 2.} Now we show that for any $x\in \Ie$, $\dist(x, \Is)\le \tau/K$. We proceed with proof by contradiction. Assume that there exists $x\in \Ie$ violating this. Together with $\tau/K\le\frac{1}{3}$ and \Cref{lem:gaussian}, we have
\begin{equation*}
\begin{aligned}
    (\phi_p*f_*)(x)&=\sum_{s=1}^S w_s\phi_p(x-x_s)+(\phi_p*r)(x)\le \phi_p(\tau/K)+\omega\phi_p(0)\\
    &\le (1+\kappa)\frac{K}{\sigma}\exp(-\pi\frac{\tau^2}{\sigma^2}) + \omega\phi_p(0)\le (1+\kappa)\phi_p(0)\exp(-\pi\frac{\tau^2}{\sigma^2}) + \omega\phi_p(0).
\end{aligned}
\end{equation*}
Therefore
\begin{equation}\label{eq:step2gauss}
\begin{aligned}
  \frac{6\beta+5\omega}{11}\phi_p(0)&\le \left|\sum_{|k|\le K}y(k)\hat{\phi}_p(k)e^{2\pi\i kx}\right| \le (\phi_p*f_*)(x)+ \left|\sum_{|k|\le K}(k)\hat{\phi}_p(k)e^{2\pi\i kx}-(\phi_P*f_*)(x)\right|\\
&<(1+\kappa)\phi_p(0)\exp(-\pi\frac{\tau^2}{\sigma^2}) + \omega\phi_p(0) + (\eps + \exp(-\pi\sigma^2))\phi_p(0)\\
&\le \left((1+\kappa)\frac{\beta-\omega}{12} + \omega + \frac{\beta-\omega}{3} + \frac{\beta-\omega}{12}\right)\phi_p(0)\\
&<\frac{6\beta+5\omega}{11}\phi_p(0),
\end{aligned}
\end{equation}
which is a contradiction.
\end{proof}

\bibliographystyle{abbrv}

\bibliography{ref}

\end{document}